\documentclass[letter]{amsart}
\usepackage{amsmath,amssymb,amscd}
\usepackage{graphicx} 
\usepackage{epsfig}
\usepackage[dvips]{color}
\usepackage{enumerate}
\usepackage{cite} 
\usepackage[left=1in,right=1in,top=1.5in,bottom=1in,footskip=.25in]{geometry}

\newtheorem{theorem}{Theorem}[section]

\newtheorem{proposition}[theorem]{Proposition}
\newtheorem{corollary}[theorem]{Corollary}

\newtheorem{question}[theorem]{Question}

\theoremstyle{definition}

\newtheorem{example}[theorem]{Example}

\newtheorem{remark}[theorem]{Remark}

\newcommand{\comment}[1]{}
\newcommand{\lk}{\ell k}
\newcommand{\bdry}{\ensuremath{\partial}}

\newcommand{\nbhd}{\ensuremath{\mathcal{N}}}

\newcommand{\Q}{\ensuremath{\mathbb{Q}}}

\newcommand{\Z}{\ensuremath{\mathbb{Z}}}

\begin{document}
\baselineskip 14pt

\title[Non-characterizing slopes for hyperbolic knots]{Non-characterizing slopes for hyperbolic knots}

\author{Kenneth L. Baker and Kimihiko Motegi}

\address{Department of Mathematics, University of Miami, 
Coral Gables, FL 33146, USA}
\email{k.baker@math.miami.edu}
\address{Department of Mathematics, Nihon University, 
3-25-40 Sakurajosui, Setagaya-ku, 
Tokyo 156--8550, Japan}
\email{motegi@math.chs.nihon-u.ac.jp}

\thanks{The first named author was partially supported by a grant from the Simons Foundation (\#209184 to Kenneth L.\ Baker).  
The second named author has been partially supported by Japan Society for the Promotion of Science, 
Grants--in--Aid for Scientific Research (C), 26400099 and Joint Research Grant of Institute of Natural Sciences at Nihon University for 2015. 
}

\dedicatory{}

\begin{abstract}
A non-trivial slope $r$ on a knot $K$ in $S^3$ is called a characterizing slope 
if whenever the result of $r$--surgery on a knot $K'$ is orientation preservingly homeomorphic to the result of $r$--surgery on $K$, 
then $K'$ is isotopic to $K$.  
Ni and Zhang ask: 
for any hyperbolic knot $K$, 
is a slope $r = p/q$ with $|p| + |q|$ sufficiently large a characterizing slope? 
In this article we answer this question in the negative by 
demonstrating that there is a hyperbolic knot $K$ in $S^3$ which has infinitely many non-characterizing slopes. 
As the simplest known example, 
the hyperbolic knot $8_6$ has no integral characterizing slopes.
\end{abstract}

\maketitle

{
\renewcommand{\thefootnote}{}
\footnotetext{2010 \textit{Mathematics Subject Classification.}
Primary 57M25
\footnotetext{ \textit{Key words and phrases.}
Dehn surgery, characterizing slope}
}
}
\section{Introduction}
\label{section:Introduction}

For a given knot $K$ in the three sphere $S^3$, 
we call $p/q \in \mathbb{Q}$ a \textit{characterizing slope} for $K$ if whenever the result of $p/q$--surgery on a knot $K'$ in $S^3$ is orientation preservingly homeomorphic to the result of $p/q$--surgery on $K$, 
then $K'$ is isotopic to $K$. 
For the trivial knot, 
Gordon \cite{Go} conjectured that every non-trivial slope $p/q \in \mathbb{Q}$ is a characterizing slope. 
Kronheimer, Mrowka, Ozsv\'ath and Szab\'o \cite{KMOS}  
proved this conjecture in the positive using Seiberg-Witten monopoles.  
See \cite{OS_genus} and \cite{OS_rational} for alternative proofs using Heegaard Floer homology. 
Furthermore, Ozsv\'ath and Szab\'o \cite{OS_trefoil_figeight} 
showed that for the trefoil knot and the figure-eight knot, 
every non-trivial slope is a characterizing slope.

On the other hand, it is known that many knots have non-characterizing slopes.  
The first such example was given by Lickorish \cite{Lic}. 
Some torus knots have non-characterizing slopes. 
For instance, $21$--surgeries on $T_{5, 4}$ and $T_{11, 2}$ produce the same oriented $3$--manifold, 
and hence $21$ is a non-characterizing slope for both $T_{5, 4}$ and $T_{11, 2}$ \cite{NZ}. 
However, 
Ni and Zhang \cite{NZ} prove that for a torus knot $T_{r, s}$ with $r > s> 1$, 
a slope $p/q$ is a characterizing slope if $p/q > 30(r^2 -1)(s^2 -1)/67$. 
This suggests that for a given knot $K$, sufficiently large slopes are characterizing ones.  
For hyperbolic knots, Ni and Zhang ask the following: 

\begin{question}[Ni and Zhang \cite{NZ}]
\label{NiZhang}
Let $K$ be a hyperbolic knot. 
Is a slope $r = p/q$ with $|p| + |q|$ sufficiently large a characterizing slope of $K$? 
Equivalently, are there only finitely many non-characterizing slopes of $K$?
\end{question}

The purpose in this article is to answer Question~\ref{NiZhang} in the negative. 

\begin{theorem}
\label{non_charct_slope}
There exists a hyperbolic knot which has infinitely many non-characterizing slopes. 
\end{theorem}

To prove Theorem~\ref{non_charct_slope} we first give a general principle to produce 
knots with infinitely many non-characterizing slopes; 
see Theorem~\ref{link_surgery}. 
(In fact, Corollary~\ref{cor:generalsurg} shows this principle produces infinitely many such knots.) 
Then we apply this to present explicit examples. 
Recall that every non-trivial slope is a characterizing slope for a trefoil knot and the figure-eight knot \cite{OS_trefoil_figeight}, 
which are genus one, fibered knots. 
If we drop one of these conditions, we have:

\begin{example}
\label{9_42_P(-3,3,5)}\
\begin{enumerate}
\item
Every integer except possibly $2$ is not a characterizing slope for the knot $9_{42}$ in Rolfsen's table. 
The knot $9_{42}$ is a fibered knot, but it has genus two.

\item
Every integer except possibly $0$ is not a characterizing slope for the pretzel knot $P(-3,3,5)$. 
The pretzel knot $P(-3,3,5)$ is a genus one knot, but it is not fibered. 
\end{enumerate}
\end{example}

\begin{question}
Is $2$ a characterizing slope for $9_{42}$?  Is $0$ a characterizing slope for $P(-3,3,5)$?
\end{question}

A modification of the above example leads us to demonstrate: 

\begin{theorem}
\label{all integers}
There exists a hyperbolic knot for which every integral slope is a non-characterizing slope. 
In particular, 
every integral slope is not a characterizing slope for the 
hyperbolic $8$--crossing knot $8_6$ in Rolfsen's table. 
\end{theorem}

\medskip

Further modifications produce the same result for prime satellite knots and composite knots.

\begin{theorem}
\label{satellite_composite}\
\begin{enumerate}
\item
Given a non-trivial knot $k$, 
there exists a prime satellite knot with $k$ a companion knot for which every integral slope is a non-characterizing slope. 

\item
Given a non-trivial knot $k$, 
there exists a composite knot with $k$ a connected summand for which every integral slope is a non-characterizing slope. 
\end{enumerate}
\end{theorem}

\medskip

Among known examples, 
 the knot $8_6$ is the simplest knot (with respect to crossing numbers) which has infinitely many non-characterizing slopes. 
So we would like to ask: 

\begin{question}
\label{ques:smallcrossing}
Are there any knots of crossing number less than $8$ that have infinitely many non-characterizing slopes?
\end{question}

\bigskip

\section{Non-characterizing slopes and twist families of surgeries}
\label{twist}

\subsection{General construction}
\label{general}

 In this subsection
we establish the following general principle.

\begin{theorem}
\label{link_surgery}
Let $k \cup c$ be a two-component link in $S^3$ such that $c$ is unknotted.
Suppose that $(0, 0)$--surgery on $k \cup c$ results in $S^3$. 
Let $K$ be the knot in $S^3$ which is surgery dual to $c$, the image of $c$,  
in the surgered $S^3$, and let $k_n$ be the knot obtained from $k$ by twisting $n$ times along $c$.
Then $K(n) \cong k_n(n)$ for all integers $n$.

Moreover, 
if $c$ is not a meridian of $k$, 
then $K \not\cong k_n$ for all but finitely many integers $n$. 
\end{theorem}

\begin{proof}
Since $(0,0)$--surgery on $k \cup c$ is $S^3$, a homology calculation shows that $|\lk(k,c)| = 1$. 
Performing $(-1/n)$--surgery along $c$ takes the knot $k$ with the surgery slope $0$ to a knot $k_n$ with a surgery slope 
$n = 0 + n (\lk(k,c))^2$, 
i.e. $n$--twist along $c$ converts a knot-slope pair $(k, 0)$ into another knot-slope pair $(k_n, n)$; 
thus we obtain a twist family of knot-slope pairs $\{ (k_n, n) \}$. 
Let $V$ be the solid torus $S^3 - \nbhd(c)$ which contains $k$ in its interior. 
Observe that $V(k; 0) \cong V(k_n; n)$ for all $n$. 

Let $(\mu_c, \lambda_c)$ be a preferred meridian-longitude pair of $c \subset S^3$, 
oriented with the right-handed orientation (so that if $c$ is oriented in the same direction as $\lambda_c$ in $\nbhd(c)$, 
then $\lk(\mu_c,c)=1$). 
Note that $\lambda_c$ represents the $0$--slope on $N(c)$ and $\lambda_c$ bounds a meridian disk of the solid torus $V$. 
Let $c_n$ be the surgery dual to the $(-1/n)$--surgery on $c$  (i.e.\ a core of the filled solid torus) with meridian $\mu_n$, 
the $(-1/n)$--surgery slope of $c$ in $\bdry V$.  
These curves $\mu_n$ are each longitudes of $V$ and satisfy 
$[\mu_n] = -[\mu_c] + n [\lambda_c] \in H_1(\bdry V)$; 
$[\mu_0] = -[\mu_c]$. 

Since $k$ wraps algebraically once in $V$, 
a preferred longitude of $k \subset V \subset S^3$ is homologous to $\mu_c$ in $V - \nbhd(k)$. 
Hence $\mu_c$ is null-homologous in $V(k; 0)$. 

Let $K$ be the surgery dual to $c$ with respect to $\lambda_c$--surgery. 
(Adapting the above notation $K$ may be regarded as $c_{\infty}$.) 
Since $(0, 0)$--surgery on $k \cup c$ results in $S^3$, 
$K$ is a knot in this surgered $S^3$ with exterior $S^3 - \nbhd(K) = V(k; 0)$ and meridian $\lambda_c$. 
Because $\mu_c$ is null-homologous in $V(k; 0)$, 
$\mu_c$ is the boundary of a Seifert surface for $K$. 

With right-handed orientation, 
a preferred meridian-longitude pair for $K$ in $S^3$ is  given by 
$(\lambda_c, -\mu_c)$.   
Thus 
$[\mu_n] = -[\mu_c] + n [\lambda_c] = n [\lambda_c] + (-[\mu_c])$, 
which corresponds to a slope $n$ with respect to the preferred meridian-longitude pair $(\lambda_c, -\mu_c)$. 
Therefore $k_n(n) = K(n)$ for all integers $n$. 

If $c$ is not a meridian of $k$, 
since $\lk(k, c) \ne 0$, 
any disk bounded by $c$ intersects $k$ more than once. 
Then it follows from \cite{KMS} that there are only finitely many $n$ such that $k_n$ is isotopic to $K$. 
\end{proof}

\medskip

\begin{remark}
\label{GM}
Gompf-Miyazaki had previously utilized the mirror of the knot $K$ associated to $k$ as described in Theorem~\ref{link_surgery} for a satellite construction of ribbon knots that generalizes the connected sum of a knot and its mirror \cite{GM}.
\end{remark}

\medskip

Let $k \cup c$ be a link as in Theorem~\ref{link_surgery} where $c$ is an unknot such that the result of $(0,0)$--surgery on $k \cup c$ is $S^3$ with surgery dual link $C \cup K$ where $K$ is dual to $c$ and $C$ is dual to $k$. 
After $0$--surgery on $c$, 
$k$ becomes some knot in $c(0) = S^1 \times S^2$. 
Since a non-trivial surgery (corresponding to the $0$--surgery) on $k \subset S^1 \times S^2$ yields $S^3$, 
due to Gabai \cite[Corollary~8.3]{gabai-propertyR}, 
it turns out that $k (\subset S^1 \times S^2)$ is an $S^1$ fiber in some product structure of $S^1 \times S^2$. 
Since the product structure of $S^1 \times S^2$ is unique up to isotopy, 
$k$ is ambient isotopic to an $S^1$ fiber in the original product structure of $c(0) = S^1 \times S^2$. 
Thus the surgery dual $C$ to $k$ in $(k \cup c)(0, 0) = S^3$ is an unknot while 
the surgery dual $K$ to $c$ is not necessarily unknotted in this $S^3$. 

Further, if $c$ is a meridian of $k$, 
then after we straighten $k$ in $c(0) = S^1 \times S^2$, 
the image $K$ of $c$ in $(k \cup c)(0, 0) = S^1 \times S^2$ 
intersects $\{ x \} \times S^2$ once for some $x \in S^1$. 
This implies that the dual $C$ to $k$ is a meridian of $K$ in $S^3$; 
see \cite[p.119]{GM}. 
Conversely, if $C$ is a meridian of $K$, 
then $c$ is a meridian of $k$. 
Thus if $c$ is not a meridian of $k$, 
then $C$ is not a meridian of $K$ neither. 

\medskip

In the proof of Theorem~\ref{link_surgery} we observe that 
$(k \cup c)(0, -\frac{1}{n}) \cong (C \cup K)(\frac{1}{0}, n)$, 
$(k \cup c)(0, -\frac{1}{n}) \cong k_n(n)$ and 
$(C \cup K)(\frac{1}{0}, n) \cong K(n)$. 
Starting with $m$--surgery instead of $0$--surgery on $k$, 
the argument in the proof of Theorem~\ref{link_surgery} 
leads us the following generalization. 
In what follows, 
$K_m$ denotes the knot obtained from $K$ by twisting $m$ times along $C$.    

\begin{corollary}
\label{cor:generalsurg}
Let $k \cup c$ be a link as in Theorem~\ref{link_surgery} with surgery dual link $C \cup K$ where $K$ is dual to $c$ and $C$ is dual to $k$. 
Then $K_m(n+m) \cong k_n(m+n)$ for any integers $m, n$.

Moreover, if $c$ is not a meridian of $k$, then each family $\{K_m\}$ and $\{k_n\}$ contains infinitely many distinct knots, each of which has only finitely many integral characterizing slopes. 
\end{corollary}

\begin{proof}
Observe that $S^3- \nbhd(k \cup c) = S^3 - \nbhd(C \cup K)$ and the meridian-longitude pairs $(\mu_k, \lambda_k)$ for $k$ and $(\mu_c, \lambda_c)$ for $c$  become meridian-longitude pairs $(\lambda_k, -\mu_k)$ for $C$ and $(\lambda_c, -\mu_c)$ for $K$.  
The latter correspondence was shown in the proof of Theorem~\ref{link_surgery}. 
For the former correspondence, 
by definition, 
$\lambda_k$ becomes a meridian of $C$, the surgery dual to $k$. 
Observe also that $\mu_k$ is homologous to $\lambda_c$, which bounds a disk of the filled solid torus 
after $0$--surgery on $c$. 
Thus $\mu_k$ is a preferred longitude of $C$. 
Now the orientation convention gives the desired result. 

Then we have the following surgery relation
\[
K_m(n+m) \cong (C \cup K)(-\tfrac{1}{m},n)
 \cong (k \cup c)(m,-\tfrac{1}{n}) 
 \cong k_n(m+n)
\]
as claimed.

If $c$ is not a meridian of $k$, then $C$ is not a meridian of $K$.  Since $\lk(k,c) \neq 0$ and $\lk(K,C)\neq 0$, 
the wrapping numbers of $k$ about $c$ and $K$ about $C$ are at least $2$. 
Then \cite[Theorem 3.2]{KMS} implies that each twist family of knots $\{k_n\}$ and $\{K_m\}$ partitions into infinitely many distinct knot types containing finitely many members.  
Therefore, since $K_m(n+m) \cong k_n(m+n)$, 
each knot in these two families has only finitely many characterizing slopes.
\end{proof}

\medskip

\begin{remark}
\label{dual link}
Let $k \cup c$ be a link in $S^3$ such that $c$ is unknotted and $(0, 0)$--surgery on $k \cup c$ yields 
$S^3$ with surgery dual link $C \cup K$. 
Then as observed above, $C$ is also unknotted and $(0, 0)$--surgery on $C \cup K$ yields $S^3$ with surgery dual $k \cup c$. 
In particular, $|\lk(K, C)| = 1$. 
\end{remark}

\subsection{Multivariable Alexander polynomials}
We take $\Delta_{A \cup B}(x,y)$ to be the {\em symmetrized} multivariable Alexander polynomial of the {\em oriented} two-component link $A \cup B$ where $x$ corresponds to the oriented meridian $\mu_A$ of $A$ and $y$ corresponds to the oriented meridian $\mu_B$ of $B$.     Due to the symmetrization, 
\[\Delta_{A \cup B}(x,y) = \Delta_{A \cup B}(x^{-1},y^{-1})= \Delta_{-A \cup -B}(x,y).\]
However, in general, $\Delta_{A \cup B}(x,y) \neq \Delta_{A \cup -B}(x,y)$.

\begin{proposition}
\label{prop:dualalexanderpolynomials}
Assume $k \cup c$ is an oriented two-component link with $\lk(k,c) =1$ such that $c$ is an unknot.  
Further assume $(0,0)$--surgery on $k \cup c$ results in $S^3$ with surgery dual $C \cup K$ 
where $K$ is dual to $c$ and $C$ is dual to $k$, oriented so that $\lk(K,C)=1$. 
Then $\Delta_{K \cup C}(x, y) = \Delta_{k \cup c}(x, y^{-1})$, 
equivalently $\Delta_{k \cup c}(x,y) = \Delta_{K \cup C}(x,y^{-1})$.
\end{proposition}

\begin{proof}
Let us write $\mu_J$ and $\lambda_J$ for the  meridian and preferred longitude of an oriented knot $J$ in $S^3$ which we view as oriented curves in $\bdry \nbhd(J)$ 
such that $\lk(J, \mu_J) = 1$ and $\lambda_J$ is homologous to $J$. 
Let $X = S^3-\nbhd(k \cup c)$ be the exterior of the link $k \cup c$.
Since the linking number of $k \cup c$ is $1$, in $H_1(X; \Z)$ we have that $[\mu_k] = [\lambda_c]$ and $[\mu_c] = [\lambda_k]$.
Furthermore these homologies are realized by oriented Seifert surfaces $\Sigma_c$ and $\Sigma_k$ that are each punctured once by $k$ and $c$ respectively.  
In particular, restricting to $X$, $\bdry \Sigma_c = \lambda_c - \mu_k$ and $\bdry \Sigma_k = \lambda_k -\mu_c$.

Since $K$ is the surgery dual to $c$ with respect to $0$--surgery on $c$ and $C$ is the surgery dual to $k$ with respect to $0$--surgery on $k$, 
$X = S^3 - \nbhd(K \cup C)$.    
Upon surgery, the punctured Seifert surfaces $\Sigma_k$ and $\Sigma_c$ cap off to oriented Seifert surfaces 
$\Sigma_K$ and $\Sigma_C$ respectively for $K$ and $C$.   
Using these surfaces to orient $K$ and $C$ and thus their meridians and longitudes, 
we obtain that $(\mu_K, \lambda_K) = (\lambda_c, -\mu_c)$ and $(\mu_C, \lambda_C) = (\lambda_k, -\mu_k)$.  
Therefore $[\mu_K] = [\mu_k]$ and $[\mu_C] = [\mu_c]$ in $H_1(X;\Z)$.  
However, since $[\lambda_C] = -[\mu_k] = -[\mu_K]$, we find that $\lk(K,C)=-1$.    
To orient $K$ and $C$ so that $\lk(K, C) = 1$, 
we must flip the orientation on $C$, say. 
Then for this correctly oriented $C$, 
we have $[\mu_C] = -[\mu_c]$. 
Hence $\Delta_{K \cup C}(x, y) = \Delta_{k \cup c}(x, y^{-1})$.
\end{proof}

\medskip

We recall also the following twisting formula for Alexander polynomials from \cite[Theorem 2.1]{BM}. 

\begin{proposition}[\cite{BM}]
\label{twisting_formula}
Let $k \cup c$ be an oriented two-component link such that $c$ is an unknot and $\omega = \lk(k, c) > 0$.  
Denote by $k_n$ a knot obtained from $k$ by $n$--twist along $c$. 
Then 
$\Delta_{k_n}(t) = \Delta_{k \cup c}(t, t^{n\omega})$. 
\end{proposition}

\medskip

Propositions~\ref{prop:dualalexanderpolynomials} and \ref{twisting_formula} lead us some symmetry 
among Alexander polynomials of $k_n$ and $K_n$. 

\begin{corollary}
\label{Alex_poly}
Let $k \cup c$ be a link as in Theorem~\ref{link_surgery} with surgery dual link $C \cup K$ where $K$ is dual to $c$ and $C$ is dual to $k$. 
Then for the twist families of knots $\{ k_n \}$ and $\{ K_n \}$, 
we have $\Delta_{k_n}(t) = \Delta_{K_{-n}}(t)$. 
In particular, $\Delta_{k}(t) = \Delta_{K}(t)$. 
\end{corollary}

\begin{proof}
We may orient $k$ and $c$ so that $\lk(k,c) = 1$; 
see the proof of Theorem~\ref{link_surgery}. 
Then Propositions~\ref{prop:dualalexanderpolynomials} and \ref{twisting_formula} show that 
$\Delta_{k_n}(t) = \Delta_{k \cup c}(t, t^n) = \Delta_{K \cup C}(t, t^{-n}) = \Delta_{K_{-n}}(t)$. 
In particular, 
putting $n = 0$, we have $\Delta_{k}(t) = \Delta_{K}(t)$. 
\end{proof}

\bigskip

\section{Examples}
\label{examples}
In this section we will provide examples which satisfy the condition in Theorem~\ref{link_surgery}, 
and hence Corollary~\ref{cor:generalsurg}. 
Example~\ref{9_42_P(-3,3,5)} follows from Examples~\ref{ex:K0} and \ref{ex:K1}. 
A slight modification gives a non-hyperbolic example, 
Example~\ref{satellite_composite_example} that demonstrates Theorem~\ref{satellite_composite}. 
We will make a further modification of the first example to present Example~\ref{nointegerchar} which implies Theorem~\ref{all integers}. 

\bigskip
Let us take a two component link $k \cup c$ with $|\lk(k, c)| = 1$ as in Figure~\ref{fig:kc-KC}. 
Then as shown in Figure~\ref{fig:kc-KC}, 
$(0, 0)$--surgery on $k \cup c$ yields $S^3$ and its surgery dual $C \cup K \subset S^3$.  
Thus $k \cup c$ satisfies the condition in Theorem~\ref{link_surgery}, and $K(n) \cong k_n(n)$ does hold 
for all integers $n$. 

Furthermore, orienting $k \cup c$ so that $\lk(k, c)=1$, 
one may calculate\footnote{For a computer assisted calculation,  
one may first use PLink within SnapPy \cite{snappy} to obtain a Dowker-Thistlethwaite code (DT code) for the link.  
Then the Knot Theory package \cite{knottheorypackage} for Mathematica \cite{mathematica} can produce the multivariable Alexander polynomial from the DT code.} 
the multivariable Alexander polynomial of $k \cup c$ to be 

\[
\Delta_{k \cup c}(x,y) =  - (x^{-1}-2+x)y^{-1} +1 - (x^{-1}-2+x) y.
\]

Hence by Proposition~\ref{twisting_formula} we have:
\[
\Delta_{k_n}(t) = \Delta_{k \cup c}(t, t^n) = - (t^{-1} -2 +t) t^{-n} + 1 - (t^{-1} -2 +t)t^n. \tag{$\star$}
\] 

In particular, since the Alexander polynomial of $k_n$ varies depending on $n$, 
$c$ is not a meridian of $k$.

\begin{figure}[h]
\begin{center}
\includegraphics[width=0.9\textwidth]{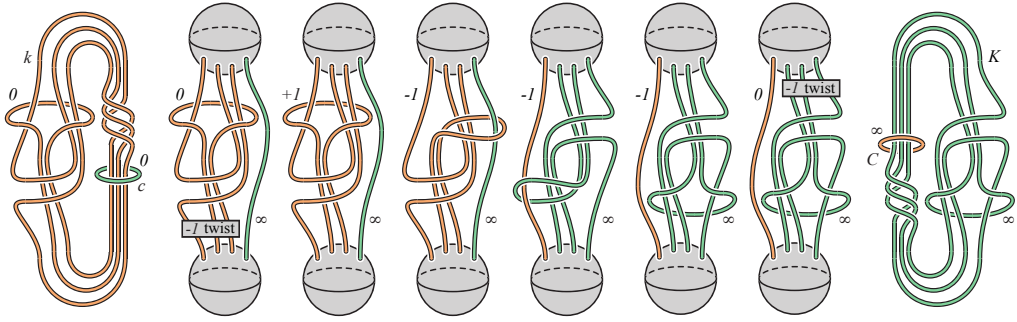}
\caption{$(0, 0)$--surgery on $k \cup c$ results in $S^3$ with its surgery dual $C \cup K$.}
\label{fig:kc-KC}
\end{center}
\end{figure}

Let us generalize this following Corollary~\ref{cor:generalsurg}. 
Let $K_m$ be a knot obtained from $K$ by $m$--twist along $C$.  
Then Corollary~\ref{cor:generalsurg} asserts that  
$K_m(n+m) \cong k_n(m+n)$ for any integers $m, n$.
Figure~\ref{fig:kcfamilies} demonstrates this fact pictorially.

\begin{figure}[h]
\begin{center}
\includegraphics[width=0.9\textwidth]{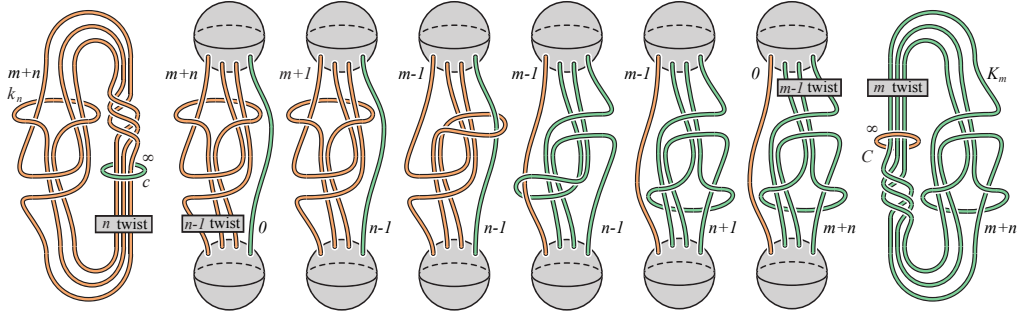}
\caption{$(m+n)$--surgery on the knot $k_n$ is equivalent to $(m+n)$--surgery on $K_m$.}
\label{fig:kcfamilies}
\end{center}
\end{figure}

Let us choose an integer $m$ arbitrarily. 
Observe that, in this example, we have $K_m = k_m$; see Figure~\ref{fig:kcfamilies}. 
Hence, if $k_n = K_m$ for some integer $n$, 
then $k_n = k_m$. 
Thus $\Delta(k_n) \doteq \Delta(k_m)$, 
and $(\star)$ implies that  $n = \pm m$. 
Thus at most $k_m$ and $k_{-m}$ can be isotopic to $K_m$. 
Since $K_m(n+m) \cong k_n(m+n)$ for all integers $m, n$, we have the following:

\begin{itemize}
\item
For a given integer $m$, 
every integral slope except possibly $0$ and $2m$ fails to be a characterizing slope for $K_m$. 
\item If furthermore $K_{-m} \neq K_{m}$, 
then $0$ will fail to be a characterizing slope as well. 
\end{itemize}

\medskip

\begin{example}[$m = 0$]
\label{ex:K0}
Let us choose $m = 0$ in the above. 
Then $K_0(n) = k_n(n)$ for all integers $n$ and, 
as mentioned above, 
every non-zero integral slope fails to be a characterizing slope for $K_0$.  
In Figure~\ref{fig:k-1topretzel} we identify $K_0 = k_0$ as the pretzel knot $P(-5,-3,3)$, 
which is known to be hyperbolic by \cite{oertel}. 
\end{example}

\begin{figure}[h]
\begin{center}
\includegraphics[width=0.9\textwidth]{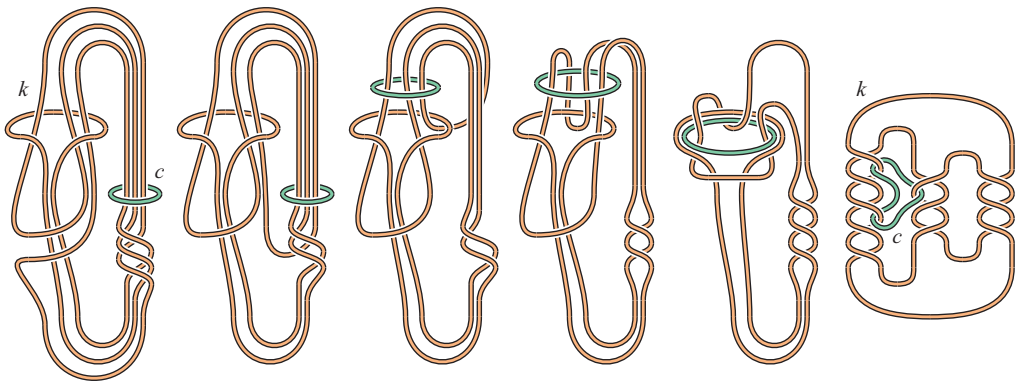}
\caption{The knot $k = k_{0}$ is isotoped into a presentation as the pretzel knot $P(-5,3,-3)$.  
The twisting circle $c$ is carried along with the isotopy.}
\label{fig:k-1topretzel}
\end{center}
\end{figure}

\begin{remark}
\label{rem:seifertsurg}
Notably, the (mirror of the) knot $P(-5,3,-3)$ was the basic example of the first two families non-strongly invertible knots  with a small Seifert fibered space surgery \cite{MMM}.    
Indeed, $(-1)$--surgery on $P(-5,3,-3)$ is the Seifert fibered space $S^2(-2/5, 3/4, -1/3)$.   

Since $P(-5,3,-3)$ is the knot $K_{0}$ 
and $K_0(n) = k_n(n)$ for all integers $n$, 
we have $K_0(-1) = k_{-1}(-1) = K_{-1}(-1)$. 
Thus $(-1)$--surgery on $K_{-1}$ is the same Seifert fibered space. 
SnapPy recognizes the complement of $K_{-1}$ as the mirror of the census manifold  $o9_{34801}$.  
Furthermore, SnapPy reports this manifold as asymmetric, 
implying that $K_{-1}$ is neither strongly invertible nor cyclically periodic, 
and hence cannot be embedded in a genus $2$ Heegaard surface. 
\end{remark}

\medskip

\begin{example}[$m = 1$]
\label{ex:K1}
By choosing $m = 1$ instead of $0$, 
we obtain a knot $K_1$ for which 
we have $K_1(n+1) = k_{n}(1+n)$ for all integers $n$. 
As we mentioned, every integral slope other than $0, 2$ are non-characterizing slope for $K_1$. 
In Figure~\ref{fig:9_42} we recognize the knot $K_{1}$ as the  $9$--crossing Montesinos knot $M(1/3, -1/2, 2/5)$ which is the 
knot $9_{42}$ in Rolfsen's table \cite{rolfsen}.  
Following \cite{oertel} $K_1$ is a hyperbolic knot. 

Now let us show that $0$--slope is also a non-characterizing slope for $K_1$. 
Since $K_1(0) \cong k_{-1}(0)$, 
it is sufficient to see that $K_1 \ne k_{-1}$. 
Recall that $K_m = k_m$ for any $m$. 
Alexander polynomials distinguish $k_1$ from $k_n$ for all $n \neq \pm 1$; see $(\star)$.  
The Jones polynomial\footnote{Kodama's software KNOT \cite{kodamasoft} was used confirm the Jones polynomials of knots.}  will however distinguish $k_1 = K_1$ and $k_{-1}$:
\[V_{k_1}(q) = q^{-3} - q^{-2} + q^{-1} -1 + q - q^2 + q^3\]
while
\[V_{k_{-1}}(q) = q^{-1} + q^{-3} - q^{-6} - q^{-8} + q^{-9} - q^{-10} + q^{-11}.\]
(As noted in Remark~\ref{rem:seifertsurg}, SnapPy also identifies the complement of $K_{-1} = k_{-1}$ as distinct from the complement of $K_1=9_{42}$, thereby distinguishing these knots.)
Hence all integers except possibly $2$ are  non-characterizing slopes for the hyperbolic knot $K_1=9_{42}$. 
\end{example}

\begin{figure}[h]
\begin{center}
\includegraphics[width=0.75\textwidth]{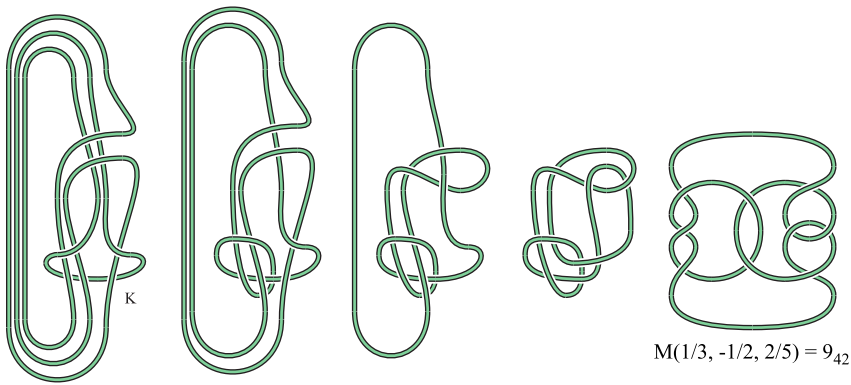}
\caption{The knot $K_1$ in Figure~\ref{fig:kcfamilies} is isotoped into a presentation as the $9$ crossing Montesinos knot $M(1/3, -1/2, 2/5)$ which may be recognized as the knot $9_{42}$ in Rolfsen's table \cite{rolfsen}. }
\label{fig:9_42}
\end{center}
\end{figure}

\medskip

Next we provide examples of non-hyperbolic knots with all integral slopes are non-characterizing slopes, 
from which Theorem~\ref{satellite_composite} follows. 

\begin{example}[Non-hyperbolic example]
\label{satellite_composite_example}
Given any non-trivial knot $k''$,  
let us take a two component link $k \cup c$ as in Figure~\ref{fig:kc_composite}, 
where $k$ is a connected sum of a knot $k'$
(which is $k$ in Figure~\ref{fig:kc-KC}, the closure of the $1$--string tangle $\tau'$) 
and the non-trivial knot $k''$ (the closure of the $1$--string tangle $\tau''$). 

\begin{figure}[h]
\begin{center}
\includegraphics[width=0.2\textwidth]{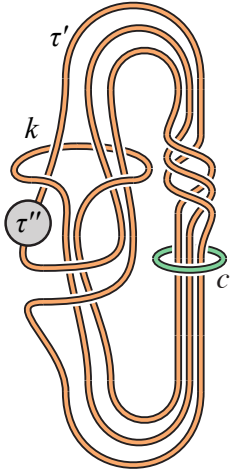}
\caption{The sum of $1$--string tangles $\tau'$ and $\tau''$ is the connected sum $k = k'\, \sharp\, k''$. 
}
\label{fig:kc_composite}
\end{center}
\end{figure}

Then as in Figure~\ref{fig:kc-KC}, we see that $(0, 0)$--surgery on $k \cup c$ gives $S^3$ with the surgery dual $C \cup K$.  
Actually, we follow the isotopy and ``light bulb" moves as indicated in Figure~\ref{fig:kc-KC} to obtain the sixth figure, 
in which $k$ is almost an $S^1$ fiber, 
but it has the connected summand $k''$ (i.e.\ the knotted arc $\tau''$).
Then we apply further ``light bulb" moves to $k$ so that it becomes an $S^1$ fiber; 
$K$ becomes a satellite knot with $k''$ as a companion knot. 
Then by Corollary~\ref{cor:generalsurg}, $K_m(n+m) \cong k_n(m+n)$ for all integers $m, n$. 

It is easy to observe that $k_n$ is a connected sum $k'_n\, \sharp\, k''$, 
where $k'_n$ is a knot obtained from $k'$ by $n$--twist along $c$. 
For instance, $k_{0} = P(-5, 3, -3)\, \sharp\, k''$ and $k_1 = 9_{42}\, \sharp\, k''$.  
Since $k'_n$ is non-trivial for all integers $n$ by $(\star)$, 
$k_n$ is not prime for all integers $n$. 

On the other hand, we show that $K_m$ is prime for all integers $m$. 
(We note that, by construction, $K_m$ has $k''$ as a companion knot for every integer $m$.)
In the following we fix an integer $m$ arbitrarily. 
First we observe that $k_n(m+n)$ is obtained by gluing $E(k'_n)$ and $E(k'')$ along their boundary tori. 
Recall that  the exterior $E(k_n)$ may be expressed as the union of the $2$-fold composing space $X$ 
(i.e. $\mbox{[disk with $2$--holes]} \times S^1$) and two knot spaces $E(k'_n)$, $E(k'')$. 
We note that $\partial X$ consists of $\partial E(k_n), \partial E(k'_n)$ and $\partial E(k'')$ 
and a regular fiber in $\partial X \cap \partial E(k_n)$ is a meridian of $k_n$. 
Since the surgery slope $m+n$ is integral, 
the corresponding Dehn filling of $X$ results in $S^1 \times S^1 \times [0, 1]$ and 
$k_n(m+n)$ can be viewed as the union of $E(k'_n)$ and $E(k'')$. 
Hence $K_m(n+m) \cong k_n(m+n) = E(k'_n) \cup E(k'')$ for all integers $n$. 
It should be noted here that $E(k'')$ is independent of $n$, but the topological type of $E(k'_n)$ depends on $n$. 
Now assume for a contradiction that $K_m$ is not prime and express  
$K_m = t_1\, \sharp\, \cdots \sharp\, t_p$ where $t_i$ is a prime knot for $1 \le i \le p$. 
Then $E(K_m)$ is the union of the $p$--fold composing space $Y = \mbox{[disk with $p$--holes]} \times S^1$ and $p$ knot spaces 
$E(t_1), \dots, E(t_p)$, 
where a regular fiber in $\partial Y \cap \partial E(K_m)$ is a meridian of $K_m$. 
Since the surgery slope $n + m$ is integral, 
the corresponding Dehn filling of $Y$ results in $(p-1)$--fold composing space $Y' = \mbox{[disk with $(p-1)$--holes]} \times S^1$. 
Hence $K_m(n+m)$ is expressed as the union $Y' \cup E(t_1) \cup \cdots \cup E(t_p)$. 
If necessary, decomposing each $E(t_i)$ further by essential tori, 
we obtain a torus decomposition of $K_m(n+m)$ in the sense of Jaco-Shalen-Johannson \cite{JS, Jo}. 
Note that identifications of $Y'$ and $E(t_i)$ ($1 \le i \le p$) depends on $n$, 
but the topological type of $E(t_i)$ ($1 \le i \le p$) does not depend of $n$. 
To make precise, let us focus on the case of $n = 0, 1$. 
Then $K_m(n+m) \cong k_n(m+n) = E(k'_n) \cup E(k'')$ and $E(k'_n)$ admits a hyperbolic structure in its interior: 
$E(k'_0)$ is the exterior of the hyperbolic knot $P(-5, 3, -3)$ and $E(k'_1)$ is the exterior of the hyperbolic knot $9_{42}$. 
If $E(k'')$ is neither hyperbolic nor Seifert fibered, 
we decompose $E(k'')$ by essential tori to obtain a torus decomposition of $K_m(n+m) \cong k_n(m+n)$ 
in the sense of Jaco-Shalen-Johannson. 
Since $E(k'_0) \not\cong E(k'_1)$, 
uniqueness of the torus decomposition of $K_m(n+m)$ shows that some $E(t_i)$ changes according as $n = 0, 1$. 
This is a contradiction. 
It follows that $K_m$ is a prime knot. 
Since $K_m$ is prime, while $k_n$ is not prime for all integers $m, n$,  
we have $\{ K_m \} \cap \{k_n \} = \emptyset$.   
Thus every integral slope fails to be a characterizing slope for a prime satellite knot $K_m$ (with a given knot $k''$ a companion knot) for any integer $m$, 
establishing Theorem~\ref{satellite_composite}(1). 
Similarly, every integral slope fails to be a characterizing slope for a composite knot $k_n$ (with a given knot $k''$ a connected summand) for any integer $n$. 
This establishes Theorem~\ref{satellite_composite}(2).

\end{example}

\begin{example}[Proof of Theorem~\ref{all integers}]
\label{nointegerchar}
Figure~\ref{fig:tangledfamily} shows a sequence of transformations relating $(m+n, \infty)$--surgery on a link $k_n \cup c$ to 
$(\infty, n+m)$--surgery on a link $C \cup K_m$.  
In particular, it gives two twist families of knots $\{k_n\}$ and $\{K_m\}$ such that $k_n(m+n) = K_m(n+m)$.

\begin{figure}[h]
\begin{center}
\includegraphics[width=0.9\textwidth]{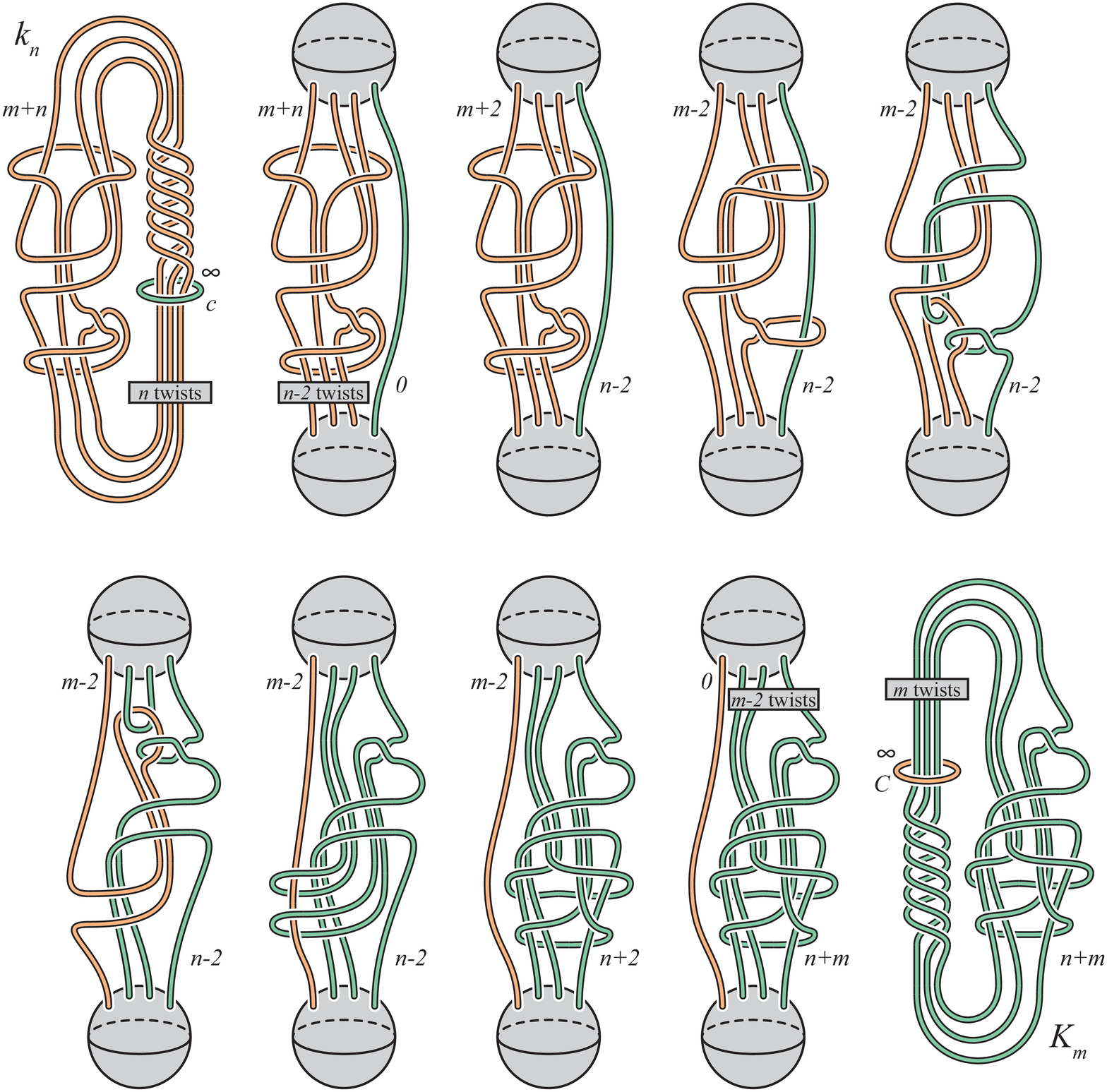}
\caption{Two families of knots  $\{k_n\}$ and $\{K_m\}$  such that $k_n(m+n) = K_m(n+m)$. 
}
\label{fig:tangledfamily}
\end{center}
\end{figure}

Replace $(m+n, \infty)$--surgery on $k_n \cup c$ by $(0, 0)$--surgery on $k_0 \cup c$, 
and follow isotopies and ``light bulb'' moves as indicated in Figure~\ref{fig:tangledfamily} to see that 
$(0, 0)$--surgery on $k_0 \cup c$ yields $S^3$ with surgery dual $C \cup K_0$ where 
$K_0$ is dual to $c$ and $C$ is dual to $k_0$.

As Figure~\ref{fig:8_6} demonstrates, 
the knot $k_1$ is the hyperbolic $8$--crossing Montesinos knot $M(3, 1/3, 1/2)$. 
It is the knot $8_6$ in Rolfsen's table,  
the two-bridge knot $\tfrac{23}{10}$. 
 Following \cite{oertel} (cf. \cite{Menasco, HatThu}) it is a hyperbolic knot.

\begin{figure}[h]
\begin{center}
\includegraphics[width=0.9\textwidth]{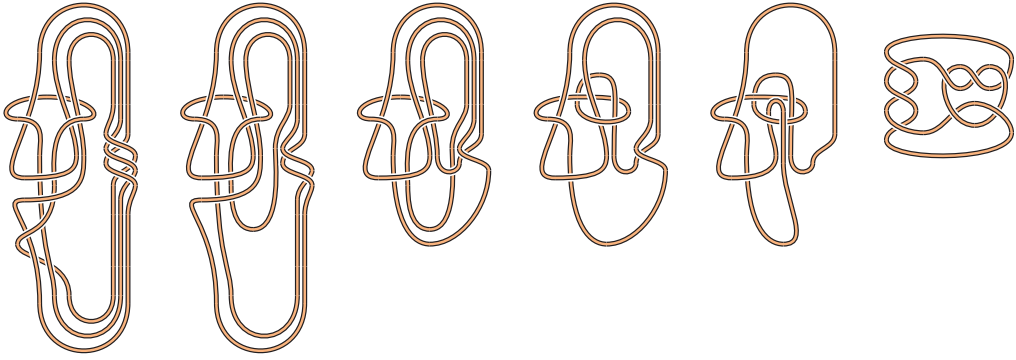}
\caption{The knot $k_1$ of the family depicted in Figure~\ref{fig:tangledfamily} is isotoped into the $8$--crossing Montesinos knot 
$M(3, 1/3, 1/2)$ which is the two-bridge knot $\tfrac{23}{10}$ and also the knot $8_6$ in Rolfsen's table. }
\label{fig:8_6}
\end{center}
\end{figure}

Using $n=0$, 
we may calculate that 
\[ 
\Delta_{k_0 \cup c}(x, y) = (x^{-1}-2+x)y^{-1} +(x^{-2} - 4x^{-1} +5 - 4x +x^2) + (x^{-1}-2+x)y, \tag{$\star\star$}
\]

which is equal to $\Delta_{K_0 \cup C}(x, y^{-1})$ by Proposition~\ref{prop:dualalexanderpolynomials}.   
Note also that $\Delta_{k_0 \cup c}(x, y) = \Delta_{k_0 \cup c}(x, y^{-1})$; see $(\star\star$).
Hence $\Delta_{k_0 \cup c}(t, t^n) = \Delta_{k_0 \cup c}(t, t^{-n}) = \Delta_{K_0 \cup C}(t, t^{n})$, 
and it follows from Proposition~\ref{twisting_formula} that 
\[\Delta_{k_n}(t) = \Delta_{K_n}(t) = (t^{-1}-2+t)t^{-n} + (t^{-2}-4t^{-1}+5-4t+t^2) + (t^{-1}-2+t)t^{n}.\] 

Thus Alexander polynomials distinguish $k_1$ from $K_m$ for all integers $m \neq \pm1$.

We further calculate the Jones polynomials of $k_1$, $K_1$, and $K_{-1}$ to be

\[V_{k_1}(q) = \frac{1}{q^7} - \frac{2}{q^6 }+ \frac{3}{q^5}- \frac{4}{q^4}+ \frac{4}{q^3} - \frac{4}{q^2}+ \frac{3}{q} - 1+ q,\]
\[V_{K_1}(q) = -\frac{1}{q^{15}}+\frac{1}{q^{14}}+\frac{1}{q^{11}}-\frac{1}{q^8}+\frac{1}{q^7}-\frac{3}{q^6}+\frac{3}{q^5}-\frac{4}{q^4}+\frac{5}{q^3}-\frac{4}{q^2}+\frac{3}{q},\]
\noindent
and 
\[V_{K_{-1}}(q) = -\frac{1}{q^{21}}+\frac{1}{q^{20}} + \frac{1}{q^{17}}- \frac{1}{q^{14}}+ \frac{1}{q^{13}} - \frac{2}{q^{12}} 
+ \frac{1}{q^{11}} - \frac{1}{q^{10}} + \frac{1}{q^9} - \frac{1}{q^8} + \frac{1}{q^7} - \frac{1}{q^6} + \frac{2}{q^5} - \frac{3}{q^4} 
+ \frac{4}{q^3} - \frac{3}{q^2} + \frac{2}{q}\]
\noindent
to conclude that $k_1 \neq K_{\pm1}$. 
Thus $k_1$ is an $8$--crossing hyperbolic knot for which every integral slope is not a characterizing slope. 
\end{example}

\bigskip

\section{Further discussions}
\label{discussion}
Let $k \cup c$ be a two-component link such that $c$ is unknotted and 
$(0, 0)$--surgery on $k \cup c$ yields $S^3$ with  surgery dual link $C \cup K$. 
Denote by $k_n$ a knot obtained from $k$ by $n$--twist along $c$, 
similarly denote by $K_m$ a knot obtained from $K$ by $m$--twist along $C$.  
Thus we obtain twist families of knots $\{ k_n \}$ and $\{ K_n \}$, 
which enjoy $K_m(n+m) \cong k_n(m+n)$ for all integers $m, n$. 
See Corollary~\ref{cor:generalsurg}.

Since $(0, 0)$--surgery on $k \cup c$ results in $S^3$, 
the linking number between $k$ and $c$ must be $\pm 1$. 
Now assume that $ k = k_0$ is an L-space knot. 
Then $k_0(m)$ is an L-space for infinitely many integers $n$ \cite[Proposition 2.1]{OSlens}, 
and since $k_0(m) = K_m(m)$ for all integers $m$, 
the twist family $\{ (K_m, m) \}$ contains infinitely many L-space surgeries. 
Furthermore, it follows from \cite[Proposition~1.10]{BM} that 
$K_m$ have the same Alexander polynomial for all $m \in \mathbb{Z}$.

Based on \cite[Conjecture~1.9]{BM}, we expect a negative answer to the following question.

\medskip

\begin{question}
\label{(0,0)-surgery_L-space knot}
Does there exist a link $k \cup c$ of an L-space knot $k$ and unknot $c$ such that $c$ is not a meridian of $k$ and $(0,0)$--surgery on $k \cup c$ is $S^3$?
\end{question}

\medskip

Recall that the non-zero coefficients of the Alexander polynomial of an L-space knot are $\pm 1$ and alternate in sign \cite[Corollary~1.3]{OSlens}. 
Hence it turns out that our knots with infinitely many non-characterizing slopes given in Section~\ref{examples} are not L-space knots.

So we may expect a positive answer to the following: 
\medskip

\begin{question}
\label{NiZhang_L-space knot}
Does an L-space knot have only finitely many non-characterizing slopes? 
\end{question}

If the answer to Question~\ref{NiZhang_L-space knot} is positive, 
then $k_0$ has only finitely many non-characterizing slopes. 
Since $k_0(m) = K_m(m)$ for all integers $m$, 
$K_m$ must be isotopic to $k_0$ except for at most finitely many integers $m$. 
Then it follows from \cite{KMS} that $c$ is a meridian of $k = k_0$. 
Thus the positive answer to Question~\ref{NiZhang_L-space knot} enables us to answer Question~\ref{(0,0)-surgery_L-space knot} in the negative. 

\medskip

More generally, we ask:

\begin{question}
For which knots $k$ does there exist a link $k \cup c$ of $k$ and an unknot $c$ such that $c$ is not a meridian of $k$ and $(0,0)$--surgery on $k \cup c$ is $S^3$?
\end{question}

Our technique cannot work for non-integral slopes. 
So we would like to propose a modified version of Ni-Zhang's question: 

\begin{question}
\label{NiZhang_modified}
For a hyperbolic knot $K$, is a non-integral slope $r = p/q$ with $|p| + |q|$ sufficiently large a characterizing slope? 
\end{question}

It should be noted here that 
Lackenby \cite{Lack} shows that for each atoroidal, homotopically trivial knot $K$ in a $3$--manifold $Y$ with 
$H_1(Y;\Q) \neq \{ 0 \}$, 
there exists a number $C(Y, K)$ such that $p/q$ is a characterizing slope for $K$ if $|q| > C(Y, K)$. 

\medskip

It is also reasonable to ask: 

\begin{question}
\label{existence}
Does every knot $K$ have a characterizing slope? 
More strongly, does every knot have infinitely many characterizing slopes? 
\end{question}

\bigskip

\bibliographystyle{plain}
\bibliography{BakerMotegi-slope}

\begin{thebibliography}{10}

\bibitem{BM}
Kenneth~L. Baker and Kimihiko Motegi.
\newblock Twist families of {L}-space knots, their genera, and {S}eifert
  surgeries, 2015.
\newblock arXiv:1506.04455.

\bibitem{knottheorypackage}
Dror Bar-Natan, Scott Morrison, and et~al.
\newblock The {M}athematica package {K}not{T}heory.
\newblock Available at
  \verb|http://katlas.org/wiki/The_Mathematica_Package_KnotTheory`| (2/5/2013).

\bibitem{snappy}
Marc Culler, Nathan~M. Dunfield, and Jeffrey~R. Weeks.
\newblock Snap{P}y, a computer program for studying the topology of
  $3$--manifolds.
\newblock Available at \texttt{http://snappy.computop.org} (20/11/2015).

\bibitem{gabai-propertyR}
David Gabai.
\newblock Foliations and the topology of {$3$}-manifolds. {III}.
\newblock {\em J. Differential Geom.}, 26(3):479--536, 1987.

\bibitem{GM}
Robert~E. Gompf and Katura Miyazaki.
\newblock Some well-disguised ribbon knots.
\newblock {\em Topology Appl.}, 64(2):117--131, 1995.

\bibitem{Go}
Cameron~McA. Gordon.
\newblock Some aspects of classical knot theory.
\newblock In {\em Knot theory ({P}roc. {S}em., {P}lans-sur-{B}ex, 1977)},
  volume 685 of {\em Lecture Notes in Math.}, pages 1--60. Springer, Berlin,
  1978.

\bibitem{HatThu}
Allen Hatcher and William Thurston.
\newblock Incompressible surfaces in {$2$}-bridge knot complements.
\newblock {\em Invent. Math.}, 79(2):225--246, 1985.

\bibitem{JS}
William Jaco and Peter Shalen.
\newblock Seifert fibered spaces in $3$--manifolds.
\newblock {\em Mem. Amer. Math. Soc.}, 21(220):viii+192, 1979.

\bibitem{Jo}
Klaus Johannson.
\newblock {\em Homotopy equivalences of $3$--manifolds with boundaries}, volume
  761 of {\em Lect. Notes in Math.}
\newblock Springer-Verlag, 1979.

\bibitem{kodamasoft}
Kouzi Kodama.
\newblock {KNOT}.
\newblock Available at \verb|http://www.artsci.kyushu-u.ac.jp/~sumi/C/knot/|
  (2012-04-03).
\newblock Port by Toshio Sumi.

\bibitem{KMS}
Masaharu Kouno, Kimihiko Motegi, and Tetsuo Shibuya.
\newblock Twisting and knot types.
\newblock {\em J. Math. Soc. Japan}, 44(2):199--216, 1992.

\bibitem{KMOS}
Peter Kronheimer, Tomasz Mrowka, Peter Ozsv{\'a}th, and Zolt{\'a}n Szab{\'o}.
\newblock Monopoles and lens space surgeries.
\newblock {\em Ann. of Math. (2)}, 165(2):457--546, 2007.

\bibitem{Lack}
Marc Lackenby.
\newblock Dehn surgery on knots in {$3$}-manifolds.
\newblock {\em J. Amer. Math. Soc.}, 10(4):835--864, 1997.

\bibitem{Lic}
W.~B.~Raymond Lickorish.
\newblock Surgery on knots.
\newblock {\em Proc. Amer. Math. Soc.}, 60:296--298 (1977), 1976.

\bibitem{MMM}
Thomas Mattman, Katura Miyazaki, and Kimihiko Motegi.
\newblock Seifert-fibered surgeries which do not arise from
  primitive/{S}eifert-fibered constructions.
\newblock {\em Trans. Amer. Math. Soc.}, 358(9):4045--4055, 2006.

\bibitem{Menasco}
William Menasco.
\newblock Closed incompressible surfaces in alternating knot and link
  complements.
\newblock {\em Topology}, 23(1):37--44, 1984.

\bibitem{NZ}
Yi~Ni and Xingru Zhang.
\newblock Characterizing slopes for torus knots.
\newblock {\em Algebr. Geom. Topol.}, 14(3):1249--1274, 2014.

\bibitem{oertel}
Ulrich Oertel.
\newblock Closed incompressible surfaces in complements of star links.
\newblock {\em Pacific J. Math.}, 111(1):209--230, 1984.

\bibitem{OS_genus}
Peter Ozsv{\'a}th and Zolt{\'a}n Szab{\'o}.
\newblock Holomorphic disks and genus bounds.
\newblock {\em Geom. Topol.}, 8:311--334, 2004.

\bibitem{OSlens}
Peter Ozsv{\'a}th and Zolt{\'a}n Szab{\'o}.
\newblock On knot floer homology and lens space surgeries.
\newblock {\em Topology}, 44(6):1281--1300, 2005.

\bibitem{OS_trefoil_figeight}
Peter Ozsvath and Zoltan Szabo.
\newblock The {D}ehn surgery characterization of the trefoil and the figure
  eight knot, 2006.
\newblock arXiv:math/0604079.

\bibitem{OS_rational}
Peter~S. Ozsv{\'a}th and Zolt{\'a}n Szab{\'o}.
\newblock Knot {F}loer homology and rational surgeries.
\newblock {\em Algebr. Geom. Topol.}, 11(1):1--68, 2011.

\bibitem{mathematica}
Wolfram Research.
\newblock Mathematica 9.0, 2013.

\bibitem{rolfsen}
Dale Rolfsen.
\newblock {\em Knots and links}, volume~7 of {\em Mathematics Lecture Series}.
\newblock Publish or Perish, Inc., Houston, TX, 1990.
\newblock Corrected reprint of the 1976 original.

\end{thebibliography}

\end{document}